\documentclass[3p,review,number,sort&compress]{elsarticle}
\usepackage{hyperref}
\usepackage{amssymb}
\usepackage{amsmath}

\newtheorem{theorem}{Theorem}[section]
\newtheorem{lemma}[theorem]{Lemma}
\newtheorem{corollary}[theorem]{Corollary}

\newproof{proof}{Proof}
\numberwithin{equation}{section}

\newcommand{\cB}{\mathcal{B}}
\newcommand{\cP}{\mathcal{P}}
\newcommand{\C}{\mathbb{C}}
\newcommand{\N}{\mathbb{N}}
\newcommand{\R}{\mathbb{R}}
\newcommand{\dR}{{\bf\dot{\R}}}
\newcommand{\eps}{\varepsilon}

\begin{document}
\begin{frontmatter}
\title{Algebras of Continuous Fourier Multipliers on Variable Lebesgue Spaces}

\author[AK]{Alexei Karlovich\corref{Alexei}}
\ead{oyk@fct.unl.pt}

\cortext[Alexei]{Corresponding author}

\address[AK]{
Centro de Matem\'atica e Aplica\c{c}\~oes,
Departamento de Matem\'atica,
Faculdade de Ci\^encias e Tecnologia,\\
Universidade Nova de Lisboa,
Quinta da Torre,
2829--516 Caparica,
Portugal}

\begin{abstract}
We show that several definitions of algebras of continuous Fourier multipliers
on variable Lebesgue spaces over the real line are equivalent under some
natural assumptions on variable exponents. Some of our results are new
even in the case of standard Lebesgue spaces and give answers on two
questions about algebras of continuous Fourier multipliers on Lebesgue spaces 
over the real line posed by H. Mascarenhas, P. Santos and M. Seidel.
\end{abstract}

\end{frontmatter}

\section{Introduction}
Let $\dR$ and $\overline{\R}$ be the compactifications of the real line $\R$
by means of the point $\infty$ and the two points $\pm\infty$, respectively.
The space of continuous functions on $\R$ that have finite limits at
$-\infty$ and $+\infty$ is denoted by $C(\overline{\R})$, and
\[
C(\dR):=\{f\in C(\overline{\R})\ :\ f(-\infty)=f(+\infty)\}.
\]
Let $\C$ stand for the constant complex-valued functions on $\R$ and
$C_0(\R)$ for the continuous functions on $\R$ which vanish at $\pm\infty$.
Notice that $\C$, $C_0(\R)$, $C(\dR)$, and $C(\overline{\R})$ are closed
subalgebras of $L^\infty(\R)$, and that $C(\dR)$ decomposes into the direct
sum $C(\dR)=\C {\bf{\dot{+}}} C_0(\R)$.

For $f\in L^1(\R)$, let $Ff$ denote the Fourier transform
\[
(Ff)(x):=\int_{\R}f(t)e^{ixt}\,dt,\quad x\in\R.
\]
If $f\in L^1(\R)\cap L^2(\R)$, then $Ff\in L^2(\R)$ and 
$\|Ff\|_{L^2(\R)}=\sqrt{2\pi}\|f\|_{L^2(\R)}$.  Since $L^1(\R)\cap L^2(\R)$ 
is dense in $L^2(\R)$, the operator $F$ extends to a bounded linear operator 
of $L^2(\R)$ onto $L^2(\R)$, which will also be denoted by $F$. The inverse 
of $F$ is given by $(F^{-1}g)(t)=(2\pi)^{-1}(Fg)(-t)$ for a.e. $t\in\R$. 

Let $p(\cdot):\R\to[1,\infty]$ be a measurable a.e. finite function. By
$L^{p(\cdot)}(\R)$ we denote the set of all complex-valued functions
$f$ on $\R$ such that
\[
I_{p(\cdot)}(f/\lambda):=\int_{\R} |f(x)/\lambda|^{p(x)} dx <\infty
\]
for some $\lambda>0$. This set becomes a Banach function space when
equipped with the norm
\[
\|f\|_{p(\cdot)}:=\inf\big\{\lambda>0: I_{p(\cdot)}(f/\lambda)\le 1\big\}.
\]
It is easy to see that if $p$ is constant, then $L^{p(\cdot)}(\R)$ is
nothing but the standard Lebesgue space $L^p(\R)$. The space
$L^{p(\cdot)}(\R)$ is referred to as a variable Lebesgue space.

Let $\cP(\R)$ denote the set of all measurable a.e. finite functions 
$p(\cdot):\R\to[1,\infty]$ such that 
\[
1<p_-:=\operatornamewithlimits{ess\,inf}_{x\in\R}p(x),
\quad
\operatornamewithlimits{ess\,sup}_{x\in\R}p(x)=:p_+<\infty.
\]
If $p(\cdot)\in\cP(\R)$, then the space $L^{p(\cdot)}(\R)$ is separable and 
reflexive, and the set $C_c^\infty(\R)$ of all infinitely differentiable 
compactly supported functions is dense $L^{p(\cdot)}(\R)$ (see, e.g., 
\cite[Chap.~2]{CF13} or \cite[Chap.~3]{DHHR11}).

Let $p(\cdot)\in\cP(\R)$. A function $a\in L^\infty(\R)$ is called a Fourier 
multiplier on the variable Lebesgue space $L^{p(\cdot)}(\R)$ if the operator
\[
W^0(a):=F^{-1}aF
\]
maps the dense set $L^2(\R)\cap L^{p(\cdot)}(\R)$ of $L^{p(\cdot)}(\R)$ into 
itself and extends to a bounded linear operator on $L^{p(\cdot)}(\R)$. Let
$M_{p(\cdot)}$ stand for the set of all Fourier multipliers on 
$L^{p(\cdot)}(\R)$. It is easy to see that that the set $M_{p(\cdot)}$ is a 
unital normed algebra under pointwise operations and the norm
\begin{equation}\label{eq:multiplier-norm}
\|a\|_{M_{p(\cdot)}}:=\|W^0(a)\|_{\cB(L^{p(\cdot)}(\R))},
\end{equation}
where $\cB(L^{p(\cdot)}(\R))$ is the algebra of all bounded linear operators
on the variable Lebesgue space $L^{p(\cdot)}(\R)$. The closure of a set 
$\mathfrak{S}$ with respect to the norm of $M_{p(\cdot)}$ will be denoted
by $\operatorname{clos}_{M_{p(\cdot)}}(\mathfrak{S})$.

Given $f\in L_{\rm loc}^1(\R)$, the Hardy-Littlewood maximal operator is 
defined by
\[
(Mf)(x):=\sup_{J\ni x}\frac{1}{|J|}\int_J|f(t)|\, dt,
\quad x\in\R,
\]
where the supremum is taken over all finite intervals $J$ containing $x$.
Here $|J|$ denotes the length of the interval $J\subset\R$.
Let $\cB_M(\R)$ denote the set of all variable exponents $p(\cdot)\in\cP(\R)$
such that the Hardy-Littlewood maximal operator $M$ is bounded on the 
variable Lebesgue space $L^{p(\cdot)}(\R)$.

L.~Diening~\cite{D04} proved that if $p(\cdot)\in\cP(\R)$ satisfies
\begin{equation}\label{eq:log-Hoelder}
|p(x)-p(y)|\le\frac{c_0}{\log(e+1/|x-y|)}, \quad x,y\in\R,
\end{equation}
for some constant $c_0>0$ and $p(\cdot)$ is constant outside some ball, then 
$p(\cdot)\in\cB_M(\R)$. Further, the behavior of $p(\cdot)$ at infinity was 
relaxed by D.~Cruz-Uribe, A.~Fiorenza, and C.~Neugebauer \cite{CFN03,CFN04}
(see also \cite[Theorem~3.16]{CF13}),  who showed that if a variable
exponent $p(\cdot)\in\cP(\R)$ satisfies \eqref{eq:log-Hoelder} and there 
exist constants $c_\infty>0$ and $p_\infty>1$ such that
\begin{equation}\label{eq:log-Hoelder-infinity}
|p(x)-p_\infty|\le\frac{c_\infty}{\log(e+|x|)},\quad x\in\R,
\end{equation}
then $p(\cdot)\in\cB_M(\R)$. Following  \cite[Section~2.1]{CF13} and 
\cite[Section~4.1]{DHHR11}, we will say that if $p(\cdot)\in\cP(\R)$ 
and conditions \eqref{eq:log-Hoelder}--\eqref{eq:log-Hoelder-infinity} are 
fulfilled, then $p(\cdot)$ is globally log-H\"older continuous. The class
of all globally log-H\"older continuous exponents will be denoted by $LH(\R)$.

Observe, however, that A.~Lerner \cite{L05} (see also \cite[Example 4.68]{CF13}
and \cite[Example~5.1.8]{DHHR11}) constructed exponents $p(\cdot)\in\cP(\R)$
discontinuous at zero or at infinity and such that, nevertheless, 
$p(\cdot)\in\cB_M(\R)$. Thus neither \eqref{eq:log-Hoelder} nor 
\eqref{eq:log-Hoelder-infinity} is necessary for $p(\cdot)\in\cB_M(\R)$. For 
more information on the class $\cB_M(\R)$ we refer to
\cite[Chaps.~2--3]{CF13} and \cite[Chaps.~4--5]{DHHR11}.

Suppose that $a:\R\to\C$ is a function of finite total variation $V(a)$ given 
by
\[
V(a):=\sup \sum_{k=1}^n |a(x_k)-a(x_{k-1})|,
\]
where the supremum is taken over all partitions of $\R$ of the form
$-\infty<x_0<x_1<\dots<x_n<+\infty$
with $n\in\N$. The set $V(\R)$ of all functions of finite total variation
on $\R$ with the norm
\[
\|a\|_{V(\R)}:=\|a\|_{L^\infty(\R)}+V(a)
\]
is a unital non-separable Banach algebra. We will frequently use the
well-known fact that $C_c^\infty(\R)\subset V(\R)$.

It is well known that the Cauchy singular integral operator $S$ given 
for a function $f\in L^1_{\rm loc}(\R)$ by
\[
(Sf)(x):=\frac{1}{\pi i}\int_\R\frac{f(t)}{t-x},\quad x\in\R,
\]
where the integral is understood in the principal value sense,
is a Calder\'on-Zygmund operator. Applying \cite[Corollary~6.3.10]{DHHR11}
to the operator $S$, we conclude that the operator $S$ is bounded on the
variable Lebesgue space $L^{p(\cdot)}(\R)$.
In view of \cite[Theorem~2]{K15a},
if $p(\cdot)\in\cB_M(\R)$ and $a\in V(\R)$, then $a\in M_{p(\cdot)}$ and
the Stechkin type inequality holds:
\begin{equation}\label{eq:Stechkin}
\|a\|_{M_{p(\cdot)}}\le\|S\|_{\cB(L^{p(\cdot)}(\R))}\|a\|_{V(\R)}.
\end{equation}
For standard Lebesgue spaces, proofs of the Stechkin inequality 
\eqref{eq:Stechkin} can be found, e.g., in 
\cite[Theorem~17.1]{BKS02},
\cite[Theorem~2.11]{D79},
\cite[Theorem~6.2.5]{EG77}. 

L. H\"ormander \cite{H60} was probably the first who studied algebras of 
continuous Fourier multipliers on standard Lebesgue spaces $L^p(\R)$. Let 
$\mathcal{S}(\R)$ be the Schwartz class of rapidly decreasing functions. For 
$p\in(1,\infty)$, let
\[
m_p:=\operatorname{clos}_{M_p}\big(\mathcal{S}(\R)\big)
\]
and
\begin{equation}\label{eq:Rp}
R_p:=\big\{r\in(1,\infty)\ :\ 
\left|1/r-1/2\right|>\left|1/p-1/2\right|\big\}.
\end{equation}
L. H\"ormander described the maximal ideal space of the Banach algebra $m_p$
and proved that if $r\in R_p$, then
\[
C_0(\R)\cap M_r
\subset 
m_p
\subset 
C_0(\R)\cap M_p
\]
(see \cite[Theorems~1.16--11.7]{H60}).
A. Fig\`a-Talamanca and G. Gaudry proved in \cite[Theorem~A]{FG71} that if 
$p\in(1,\infty)\setminus\{2\}$, then there exists a function in 
$C_0(\R)\cap M_p$ which is not the 
limit in $M_p$ of Fourier transforms of integrable functions, whence 
\[
m_p\subsetneqq C_0(\R) \cap  M_p.
\]

For a variable exponent $p(\cdot)\in\cB_M(\R)$, consider the following 
algebras of continuous Fourier multipliers:
\[
C_{p(\cdot)}(\dR)
:=
\operatorname{clos}_{M_{p(\cdot)}}\big(C(\dR)\cap V(\R)\big),
\quad
C_{p(\cdot)}(\overline{\R})
:=
\operatorname{clos}_{M_{p(\cdot)}}\left(C(\overline{\R})\cap V(\R)\right).
\]

For standard Lebesgue spaces, these definitions go back to R. Duduchava 
\cite[Chap.~I, \S 3.2]{D79} and I. B. Simonenko, C. N. Min 
\cite[p.~50]{SM86}, respectively (see also \cite[pp.~325, 331]{BKS02}).

Let $\C{\bf{\dot{+}}} C_c^\infty(\R)$ denote the set of functions of the 
form $f=c+\varphi$,
where $c\in\C$ and $\varphi\in C_c^\infty(\R)$.
Our first main result is the following.
\begin{theorem}\label{th:first}
If $p(\cdot)\in\cB_M(\R)$, then
\[
C_{p(\cdot)}(\dR)
=
C(\dR)\cap C_{p(\cdot)}(\overline{\R})
=
\operatorname{clos}_{M_{p(\cdot)}}
\big(\C {\bf{\dot{+}}} C_c^\infty(\R)\big).
\]
\end{theorem}

For standard Lebesgue spaces $L^p(\R)$, $1<p<\infty$, the first equality is 
proved in \cite[Lemma~3.1(i)]{KS94}. We were not able to find the second 
equality stated explicitly for standard Lebesgue spaces, however it is 
implicit, e.g., in the proof \cite[Lemma~3.8]{BBK04}.

For $p(\cdot)\in\cP(\R)$, let
\begin{equation}\label{eq:range-theta}
\theta_{p(\cdot)}:=\min\big\{1,2/p_+,2-2/p_-\big\}.
\end{equation}
It follows from \cite[Corollary~2.3]{RS08} that if $p(\cdot)\in\cP(\R)$ and 
$\theta\in(0,\theta_{p(\cdot)})$, then the variable exponent 
$p_\theta(\cdot):\R\to[1,\infty]$ defined by
\begin{equation}\label{eq:p-theta}
\frac{1}{p(x)}=\frac{\theta}{2}+\frac{1-\theta}{p_\theta(x)},
\quad x\in\R,
\end{equation}
belongs to $\cP(\R)$.

Further, by $\cB_M^*(\R)$ denote the set of all variable exponents 
$p(\cdot)\in\cB_M(\R)\setminus\{2\}$ for which there exists
$\tau_{p(\cdot)}\in(0,\theta_{p(\cdot)}]$ depending on $p(\cdot)$
such that $p_\theta(\cdot)\in\cB_M(\R)$ for all 
$\theta\in(0,\tau_{p(\cdot)})$, where 
$\theta_{p(\cdot)}$ is defined by
\eqref{eq:range-theta} and $p_\theta(\cdot)$ is defined by \eqref{eq:p-theta}.
Note that in view of Lemma~\ref{le:log-Hoelder} below, we have
\[
LH(\R)\subset\cB_M^*(\R).
\]
Therefore, the class $\cB_M^*(\R)$ contains many non-trivial 
variable exponents.

For $p(\cdot)\in\cB_M^*(\R)$, consider
\[
M_{\langle p(\cdot)\rangle}:=
\bigcup_{\theta\in(0,\tau_{p(\cdot)})}M_{p_\theta(\cdot)}.
\]
By analogy with H. Mascarenhas, P. Santos and M. Seidel
\cite[p.~955]{MSS14}, \cite[p.~89]{MSS17},
consider the following algebras of continuous Fourier multipliers
on variable Lebesgue spaces $L^p(\cdot)(\R)$:
\[
C_{\langle p(\cdot)\rangle}(\dR)
:=
\operatorname{clos}_{M_{p(\cdot)}}
\big(C(\dR)\cap M_{\langle p(\cdot)\rangle}\big),
\quad
C_{\langle p(\cdot)\rangle}(\overline{\R})
:=
\operatorname{clos}_{M_{p(\cdot)}}
\left(C(\overline{\R})\cap M_{\langle p(\cdot)\rangle}\right).
\]

Note that if $p\in(1,\infty)\setminus\{2\}$ is constant, then it is not
difficult to see that
\begin{equation}\label{eq:M-p-cloud}
M_{\langle p\rangle}=\bigcup_{r\in R_p} M_r,
\end{equation}
where $R_p$ is given by \eqref{eq:Rp}. Then $C_{\langle p\rangle}(\dR)$
and $C_{\langle p\rangle}(\overline{\R})$ coincide with the algebras
denoted in \cite[p.~955]{MSS14}, \cite[p.~89]{MSS17} by
$C(\dR)_p$ and $C(\overline{\R})_p$, respectively. 

Our main motivation for this work are the questions
whether $C_p(\dR)$ coincides with $C_{\langle p\rangle}(\dR)$ and 
whether $C_{\langle p\rangle}(\dR)$ is a proper subset of $C(\dR)\cap M_p$
for a constant exponent $p\in(1,\infty)\setminus\{2\}$, which were posed by 
H.~Mascarenhas, P.~Santos and M.~Seidel (see \cite[Remark~25]{MSS14} and 
\cite[Remark~1~(iii)]{MSS17}, respectively). Note that the positive answer 
on the first question for the discrete analogues of $C_p(\dR)$ and
$C_{\langle p\rangle}(\dR)$ is contained in
\cite[Propositions~2.45 and 6.8(b)]{BS06}.

Our second main result is the following.
\begin{theorem}\label{th:second}
If $p(\cdot)\in\cB_M^*(\R)$, then 
\[
C_{p(\cdot)}(\dR)=C_{\langle p(\cdot)\rangle}(\dR),
\quad
C_{p(\cdot)}(\overline{\R})=C_{\langle p(\cdot)\rangle}(\overline{\R}).
\]
\end{theorem}

The collection of all functions of the form $c+Ff$ with $c\in\C$ and 
$f\in L^1(\R)$ is denoted by $W(\R)$. It is well known that $W(\R)$ is a
Banach algebra with respect to pointwise operations and the norm
\[
\|c+Ff\|_{W(\R)}:=|c|+\|f\|_{L^1(\R)}.
\]
The algebra $W(\R)$ is usually called the Wiener algebra.
By \cite[Theorem~2.1]{D79}, if $a\in W(\R)$ and $1<p<\infty$, then 
$\|a\|_{M_p}\le\|a\|_{W(\R)}$.

As a corollary of our results, the equality 
$C_p(\dR)=\operatorname{clos}_{M_p}\big(W(\R)\big)$
(which is implicit in \cite[Chap.~I, \S~3.2]{D79}) and a deep result by 
A.~Fig\`a-Talamanca and G.~Gaudry \cite[Theorem~A]{FG71}, we obtain the 
following results. In particular, they give the answers on the above 
mentioned questions posed in \cite{MSS14,MSS17}.
\begin{corollary}\label{co:standard-Lp}
If $p\in(1,\infty)\setminus\{2\}$, then
\[
C_p(\dR)=\operatorname{clos}_{M_p}\big(W(\R)\big)=
C_{\langle p\rangle}(\dR)\subsetneqq C(\dR)\cap M_p,
\quad
C_p(\overline{\R})=
C_{\langle p\rangle}(\overline{\R}).
\]
\end{corollary}

The paper is organized as follows. We start preliminaries collected in
Section~\ref{sec:preliminaries} with the proof of the embedding 
$LH(\R)\subset\cB_M^*(\R)$. Further, we state the interpolation theorem of
the Riesz-Thorin type for variable exponent spaces and a remarkable result
by L.~Diening, which can rephrased on the interpolation language as follows:
for every variable Lebesgue space $L^{p(\cdot)}(\R)$ with 
$p(\cdot)\in\cB_M(\R)$ there exists a standard Lebesgue space $L^{p_0}(\R)$
with $1<p_0<\infty$ and a variable Lebesgue space $L^{p_\theta(\cdot)}(\R)$
with $p_\theta(\cdot)\in\cB_M(\R)$ and \textit{some} $\theta\in(0,1)$
such that $L^{p(\cdot)}(\R)$ is an interpolation space between 
$L^{p(\cdot)}(\R)$ and $L^{p_\theta(\cdot)}(\R)$. Note that in the definition 
of the class $\cB_M^*(\R)$, we require that $p_0=2$ and $p_\theta\in\cB_M(\R)$ 
for \textit{all} $\theta\in(0,\tau_{p(\cdot)})$. We conclude these 
preliminaries with two results, which follow from our work \cite{KS-PLMS}: 
the continuous embedding of $M_{p(\cdot)}$ into $L^\infty(\R)$ with the 
optimal embedding constant $1$ and the inequality 
$\|a*\varphi_\delta\|_{M_{p(\cdot)}}\le\|a\|_{M_{p(\cdot)}}$, where
$\varphi_\delta$ is a mollifier and $a\in M_{p(\cdot)}$.

In Section~\ref{sec:proofs}, we first prove the key theorem saying that a 
continuous function $a$ vanishing at infinity belongs to the closure of 
$C_c^\infty(\R)$ in the multiplier norm of $M_{p(\cdot)}$ if it satisfies 
the condition $a\in M_{\langle p(\cdot)\rangle}$ if $p(\cdot)\in\cB_M^*(\R)$
and $a\in V(\R)$ if $p(\cdot)\in\cB_M(\R)$. The proof of this result
follows ideas of L.~H\"ormander \cite[Theorem~1.16]{H60} and is based on 
the interpolation between $L^{p_\theta(\cdot)}(\R)$ and $L^2(\R)$. 
Armed with the above mentioned key theorem, we give proofs of our main 
results (Theorems~\ref{th:first} and \ref{th:second}) and its 
Corollary~\ref{co:standard-Lp}.

In Section~\ref{sec:final}, we discuss several problems related to our main 
results. First we observe that $C_{p(\cdot)}(\dR)\subset SO_{p(\cdot)}$,
where $SO_{p(\cdot)}$ is an algebra of slowly oscillating at infinity
Fourier multipliers.
We show that $\operatorname{clos}_{M_{p(\cdot)}}\big(V(\R)\big)=
\operatorname{clos}_{M_{p(\cdot)}}\big(P\C^0\big)$ if $p(\cdot)\in\cB_M(\R)$, 
where $P\C^0$ is the set of all piecewise constant functions with finite 
number of jumps. This closure is denoted by $PC_{p(\cdot)}$ because 
$V(\R)\subset PC$, where $PC$ is the set of all piecewise continuous 
functions (having at most countable set of jumps). It is natural to 
conjecture that $PC_{p(\cdot)}$ coincides with the closure of 
$PC\cap M_{\langle p(\cdot)\rangle}$ if $p(\cdot)\in\cB_M^*(\R)$.
We have been able only to confirm a weaker version of this conjecture: 
$PC_{p(\cdot)}$ coincides with the closure of 
$PC^0\cap M_{\langle p(\cdot)\rangle}$ if $p(\cdot)\in\cB_M^*(\R)$, where 
$PC^0$ stands for the set of all functions in $PC$ with finite sets of jumps. 
Finally, we shortly discuss a gap in the proof of the embedding
$C(\overline{\R})\cap\operatorname{clos}_{M_p}\big(M_{\langle p\rangle}\big)
\subset C_p(\overline{\R})$ in \cite[Lemma~3.1(ii)]{KS94}, which we have not
been able to fill in, and state as an open problem whether the classes 
$\cB_M^*(\R)$ and $\cB_M(\R)$ coincide.
\section{Preliminaries}\label{sec:preliminaries}\label{sec:proofs}
\subsection{The class $LH(\R)$ is contained in $\cB_M(\R)$}
The following lemma is implicit in \cite{RS08}.
\begin{lemma}\label{le:log-Hoelder}
We have $LH(\R)\subset\cB_M^*(\R)$.
\end{lemma}
\begin{proof}
Let $p(\cdot)\in LH(\R)$. Then $p(\cdot)\in\cP(\R)$ and 
\eqref{eq:log-Hoelder}--\eqref{eq:log-Hoelder-infinity} are fulfilled. For
$\theta\in(0,\theta_{p(\cdot)})$, let $p_\theta(\cdot)$ be defined by
\eqref{eq:p-theta}. Then
\begin{equation}\label{eq:log-Hoelder-1}
p_\theta(x)=\frac{2(1-\theta)p(x)}{2-\theta p(x)},
\quad
x\in\R.
\end{equation}
By \cite[Corollary~2.3]{RS08}, $p_\theta(\cdot)\in\cP(\R)$. Put
\begin{equation}\label{eq:log-Hoelder-2}
(p_\theta)_\infty:=\frac{2(1-\theta)p_\infty}{2-\theta p_\infty},
\end{equation}
where $p_\infty$ is the constant from condition \eqref{eq:log-Hoelder-infinity}.
It follows from \eqref{eq:log-Hoelder-infinity} that
\[
p_\infty=\lim_{x\to\infty}p(x),
\]
whence $p_\infty\in[p_-,p_+]$. Therefore, taking into account 
\eqref{eq:range-theta}, we see that for $x\in\R$,
\begin{equation}\label{eq:log-Hoelder-3}
2-\theta p(x)\ge 2-\theta p_+>0,
\quad
2-\theta p_\infty\ge 2-\theta p_+>0.
\end{equation}
In view of \eqref{eq:log-Hoelder-1}--\eqref{eq:log-Hoelder-3}, we obtain for
$x,y\in\R$,
\begin{align}
\label{eq:log-Hoelder-4}
|p_\theta(x)-p_\theta(y)|
&=
\left|\frac{4(1-\theta)(p(x)-p(y) )}{(2-\theta p(x))(2-\theta p(y))}\right|
\le 
\frac{4(1-\theta)}{(2-\theta p_+)^2}|p(x)-p(y)|,
\\
\label{eq:log-Hoelder-5}
|p_\theta(x)-(p_\theta)_\infty|
&=
\left|
\frac{4(1-\theta)(p(x)-p_\infty)}{(2-\theta p(x))(2-\theta p_\infty)}
\right|
\le 
\frac{4(1-\theta)}{(2-\theta p_+)^2}|p(x)-p_\infty|.
\end{align}
It follows from \eqref{eq:log-Hoelder}--\eqref{eq:log-Hoelder-infinity}
and \eqref{eq:log-Hoelder-4}--\eqref{eq:log-Hoelder-5} that
$p_\theta(\cdot)\in LH(\R)$. Therefore $p_\theta(\cdot)\in\cB_M(\R)$
for every $\theta\in(0,\theta_{p(\cdot)})$. Thus $p(\cdot)\in\cB_M^*(\R)$.
\qed
\end{proof}
\subsection{Interpolation in variable Lebesgue spaces}
One of our main tools is the following interpolation theorem of the 
Riesz-Thorin type  for variable Lebesgue spaces (see 
\cite[Corollary~7.1.4]{DHHR11} and also \cite[Theorem~14.16]{M83}).
\begin{theorem}\label{th:interpolation}
Let $p_j(\cdot):\R\to[1,\infty]$, $j=0,1$, be a.e. finite measurable functions 
and let $p_\vartheta(\cdot):\R\to[1,\infty]$ be defined for 
$\vartheta\in(0,1)$ by
\[
\frac{1}{p_\vartheta(x)}=\frac{\vartheta}{p_0(x)}+\frac{1-\vartheta}{p_1(x)},
\quad
x\in\R.
\]
Suppose $A$ is a linear operator defined on 
$L^{p_0(\cdot)}(\R)+L^{p_1(\cdot)}(\R)$.
If $A\in\cB(L^{p_j(\cdot)}(\R))$ for $j=0,1$, then
$A\in\cB(L^{p_\vartheta(\cdot)}(\R))$ for all $\vartheta\in(0,1)$ and
\begin{equation}\label{eq:Riesz-Thorin}
\|A\|_{\cB(L^{p_\vartheta(\cdot)}(\R))} \le 4
\|A\|_{\cB(L^{p_0(\cdot)}(\R))}^\vartheta
\|A\|_{\cB(L^{p_1(\cdot)}(\R))}^{1-\vartheta}.
\end{equation}
\end{theorem}

If $p_j$, $j=0,1$, are constant, then the above result is the classical
Riesz-Thorin  interpolation theorem, and inequality \eqref{eq:Riesz-Thorin}
holds with the interpolation constant $1$ in the place of $4$.

The above interpolation theorem will be used in the case when $p_0(\cdot)$
is either equal to $2$ (that is, in the case $p(\cdot)\in\cB_M^*(\R)$) or is 
a  constant $p_0\in(1,\infty)$ determined by the following property of
$p(\cdot)\in\cB_M(\R)$, which was communicated to the authors of \cite{KS13} 
by L. Diening.
\begin{theorem}[{\cite[Theorem~4.1]{KS13}}]
\label{th:Deining-interpolation}
If $p(\cdot)\in\cB_M(\R)$, then there exist two constants $p_0\in(1,\infty)$ and
$\theta\in(0,1)$, and a variable exponent $p_\theta(\cdot)\in\cB_M(\R)$ such 
that
\begin{equation}\label{eq:Diening-interpolation}
\frac{1}{p(x)}=\frac{\theta}{p_0}+\frac{1-\theta}{p_\theta(x)},
\quad x\in\R.
\end{equation}
\end{theorem}
\subsection{Banach algebra of the Fourier multipliers on variable Lebesgue 
spaces}
The results of this subsection allow us to extend techniques working
for Fourier multipliers in the case of standard Lebesgue spaces, which are 
invariant under reflection and translations, to the case of variable Lebesgue 
spaces, which are neither reflection-invariant nor translation-invariant.
\begin{theorem}\label{th:continuous-embedding-of-multipliers}
If $p(\cdot)\in\cB_M(\R)$, then for every $a\in M_{p(\cdot)}(\R)$,
\begin{equation}\label{eq:continuous-embedding}
\|a\|_{L^\infty(\R)}\le\|a\|_{M_{p(\cdot)}},
\end{equation}
and the constant $1$ on the right-hand side is best possible.
\end{theorem}
\begin{proof}
It is well known that the Cauchy singular integral operator $S$ 
is a Calder\'on-Zygmund operator. Applying \cite[Corollary~6.3.10]{DHHR11}
to the operator $S$, we conclude that the operator $S$ is bounded on the
variable Lebesgue space $L^{p(\cdot)}(\R)$. Then, in view of 
\cite[Theorem~3.9]{KS14},
\begin{equation}\label{eq:AX}
C_{p(\cdot)}:=\sup_{-\infty<a<b<\infty}
\frac{1}{b-a}\|\chi_{(a,b)}\|_{L^{p(\cdot)}(\R)}
\|\chi_{(a,b)}\|_{L^{p'(\cdot)}(\R)}<\infty,
\end{equation}
where $1/p(x)+1/p'(x)=1$ for $x\in\R$.
If \eqref{eq:AX} is fulfilled, then inequality \eqref{eq:continuous-embedding}
follows from \cite[inequality (1.2) and Corollary~4.2]{KS-PLMS}.
\qed
\end{proof}

Inequality \eqref{eq:continuous-embedding} was established in
\cite[Theorem~1]{K15b} with the constant 
\[
\|S\|_{\cB(L^{p(\cdot)}(\R))}C_{p(\cdot)}\ge 1
\]
in the place of the optimal constant $1$ on the right-hand side of 
\eqref{eq:continuous-embedding}. We have to report that the formulation
of that theorem contains an inaccuracy (the constant $C_{p(\cdot)}$ is written 
in the denominator there, as well as, in the end of the proof of \cite{K15b},
starting from inequality (2.8)).
\begin{corollary}
If $p(\cdot)\in\cB_M(\R)$, then $M_{p(\cdot)}$ is a Banach algebra under
pointwise operations and the norm given by \eqref{eq:multiplier-norm}.
\end{corollary}

Since \eqref{eq:continuous-embedding} is available, the above result can be
proved by an easy adaptation of the arguments from the proof of 
\cite[Proposition~2.5.13]{G14} (we refer to the proof of 
\cite[Corollary~1]{K15b} for the details).
\begin{theorem}\label{th:convolution-with-multiplier}
Suppose that a non-negative even function $\varphi\in C_c^\infty(\R)$ satisfies 
the condition
\begin{equation}\label{eq:convolution-with-multiplier-1}
\int_\R\varphi(x)\,dx=1
\end{equation}
and the function $\varphi_\delta$ is defined for $\delta>0$ by 
\begin{equation}\label{eq:convolution-with-multiplier-2}
\varphi_\delta(x):=\delta^{-1}\varphi(x/\delta),
\quad x\in\R,
\quad \delta>0,
\end{equation}
If $p(\cdot)\in\cB_M(\R)$ and $a\in M_{p(\cdot)}$, then for every $\delta>0$,
\begin{equation}\label{eq:convolution-with-multiplier-3}
\|a*\varphi_\delta\|_{M_{p(\cdot)}}\le\|a\|_{M_{p(\cdot)}}.
\end{equation}
\end{theorem}
\begin{proof}
It follows from the proof of Theorem~\ref{th:continuous-embedding-of-multipliers}
and \cite[Lemma~3.3]{KS-PLMS} that the space $L^{p(\cdot)}(\R)$
satisfies the hypotheses of \cite[Theorem~1.3]{KS-PLMS}. It is shown in its
proof (see \cite[Section~4.2]{KS-PLMS}) that for every $\delta>0$  and every 
$f\in\mathcal{S}(\R)\cap L^{p(\cdot)}(\R)$, 
\[
\|F^{-1}(a*\varphi_\delta)Ff\|_{L^{p(\cdot)}(\R)}
\le 
\sup\left\{
\frac{\|F^{-1}aFf\|_{L^{p(\cdot)}(\R)}}{\|f\|_{L^{p(\cdot)}(\R)}}:
f\in\big(\mathcal{S}(\R)\cap L^{p(\cdot)}(\R)\big)\setminus\{0\}
\right\}
\|f\|_{L^{p(\cdot)}(\R)}.
\]
Then, for every $\delta>0$,
\begin{equation}\label{eq:convolution-with-multiplier-4}
\sup\left\{
\frac{\|F^{-1}(a*\varphi_\delta)Ff\|_{L^{p(\cdot)}(\R)}}{\|f\|_{L^{p(\cdot)}(\R)}}:
f\in\big(\mathcal{S}(\R)\cap L^{p(\cdot)}(\R)\big)\setminus\{0\}
\right\}
\le\|a\|_{M_{p(\cdot)}}.
\end{equation}
Since $p_+<\infty$, the space $L^{p(\cdot)}(\R)$ is separable
(see, e.g., \cite[Theorem~2.78]{CF13}). Then it follows from 
\cite[Chap.~1, Corollary~5.6]{BS88} and \cite[Theorem~2.3 and 6.1]{KS-PLMS}
that for every $\delta>0$, the left-hand side of inequality
\eqref{eq:convolution-with-multiplier-4} coincides with
$\|a*\varphi_\delta\|_{M_{p(\cdot)}}$, which completes the proof
of inequality \eqref{eq:convolution-with-multiplier-3}.
\qed
\end{proof}
\section{Proofs of the main results}
\subsection{Continuous Fourier multipliers vanishing at infinity}
Since the interpolation theorem of the Riesz-Thorin type is available for
variable Lebesgue spaces (see Theorem~\ref{th:interpolation}), we can prove the 
next key result following ideas of L.~H\"ormander from the proof of
\cite[Theorem~1.16]{H60}.
\begin{theorem}\label{th:vanishing-multipliers}
\begin{enumerate}
\item[(a)]
If $p(\cdot)\in\cB_M^*(\R)$, then
\[
C_0(\R)\cap M_{\langle p(\cdot)\rangle}
\subset
\operatorname{clos}_{M_{p(\cdot)}}\big(C_c^\infty(\R)\big).
\]
\item[(b)]
If $p(\cdot)\in\cB_M(\R)$, then 
\[
C_0(\R)\cap V(\R)
\subset
\operatorname{clos}_{M_{p(\cdot)}}\big(C_c^\infty(\R)\big).
\]
\end{enumerate}
\end{theorem}
\begin{proof}
For $n\in\N$, let
\[
\psi_n(x):=\left\{\begin{array}{lll}
1 &\mbox{if} & |x|\le n,
\\
n+1-|x| &\mbox{if} & n<|x|<n+1,
\\
0 &\mbox{if} & |x|\ge n+1.
\end{array}\right.
\]
Then $\psi_n$ has compact support and $\|\psi_n\|_{V(\R)}=3$.

(a) Let $a\in C_0(\R)\cap M_{\langle p(\cdot)\rangle}$. Fix $\eps>0$. By the 
definition of $M_{\langle p(\cdot)\rangle}$, there exists 
$\theta\in (0,\tau_{p(\cdot)})$ such that 
$a\in C_0(\R)\cap M_{p_\theta(\cdot)}$, where $p_\theta(\cdot)$ is defined by
\eqref{eq:p-theta}. By the Stechkin type inequality
\eqref{eq:Stechkin}, for every $n\in\N$,
\[
\|\psi_n\|_{M_{p_\theta(\cdot)}}\le 3\|S\|_{\cB(L^{p_\theta(\cdot)}(\R))}=:c_\theta.
\]
Take $b_n:=a\psi_n$. Then
\begin{equation}\label{eq:vanishing-multipliers-1}
\lim_{n\to\infty}\|a-b_n\|_{L^\infty(\R)}=0,
\end{equation}
$b_n\in C_0(\R)$ has compact support and for every $n\in\N$,
\begin{equation}\label{eq:vanishing-multipliers-2}
\|a-b_n\|_{M_{p_\theta(\cdot)}} 
\le 
\|a\|_{M_{p_\theta(\cdot)}}\left(1+\|\psi_n\|_{M_{p_\theta(\cdot)}}\right)
\le 
(1+c_\theta)\|a\|_{M_{p_\theta(\cdot)}}
\end{equation}
and
\begin{equation}\label{eq:vanishing-multipliers-3}
\|b_n\|_{M_{p_\theta(\cdot)}} 
\le 
\|\psi_n\|_{M_{p_\theta(\cdot)}}
\|a\|_{M_{p_\theta(\cdot)}}
\le 
c_\theta\|a\|_{M_{p_\theta(\cdot)}}.
\end{equation}
Equality \eqref{eq:p-theta}, Theorem~\ref{th:interpolation} and inequality
\eqref{eq:vanishing-multipliers-2} imply that for every $n\in\N$,
\begin{equation}\label{eq:vanishing-multipliers-4}
\|a-b_n\|_{M_{p(\cdot)}} 
\le 
4\|a-b_n\|_{L^\infty(\R)}^\theta
\|a-b_n\|_{M_{p_\theta(\cdot)}}^{1-\theta}
\le 
4(1+c_\theta)^{1-\theta}\|a\|_{M_{p_\theta(\cdot)}}^{1-\theta}
\|a-b_n\|_{L^\infty(\R)}^\theta.
\end{equation}
It follows from \eqref{eq:vanishing-multipliers-1} and
\eqref{eq:vanishing-multipliers-3} that there exists $n_0\in\N$ such that
\begin{equation}\label{eq:vanishing-multipliers-5}
\|a-b_{n_0}\|_{M_{p(\cdot)}}<\eps/2. 
\end{equation}

Let $\varphi\in C_c^\infty(\R)$ be a non-negative even function satisfying 
\eqref{eq:convolution-with-multiplier-1} and for $\delta>0$ the function
$\varphi_\delta$ be defined by \eqref{eq:convolution-with-multiplier-2}. By
Theorem~\ref{th:convolution-with-multiplier} and inequality
\eqref{eq:vanishing-multipliers-3}, for every $\delta>0$,
\begin{equation}\label{eq:vanishing-multipliers-6}
\|b_{n_0}*\varphi_\delta\|_{M_{p_\theta(\cdot)}}
\le 
\|b_{n_0}\|_{M_{p_\theta(\cdot)}} 
\le 
c_\theta\|a\|_{M_{p_\theta(\cdot)}}.
\end{equation}
It follows from \cite[Propositions~4.18, 4.20--4.21]{B11} that 
$b_{n_0}*\varphi_\delta\in C_c^\infty(\R)$ and 
\begin{equation}\label{eq:vanishing-multipliers-7}
\lim_{\delta\to 0}\|b_{n_0}*\varphi_\delta-b_{n_0}\|_{L^\infty(\R)}=0.
\end{equation}
As before, equality \eqref{eq:p-theta}, Theorem~\ref{th:interpolation} and
inequality \eqref{eq:vanishing-multipliers-6} imply that for every $\delta>0$,
\begin{align}
\|b_{n_0}*\varphi_\delta-b_{n_0}\|_{M_{p(\cdot)}} 
&\le 
4 
\|b_{n_0}*\varphi_\delta-b_{n_0}\|_{L^\infty(\R)}^\theta 
\|b_{n_0}*\varphi_\delta-b_{n_0}\|_{M_{p_\theta(\cdot)}}^{1-\theta}
\nonumber
\\
&\le 
2^{3-\theta}c_\theta^{1-\theta}\|a\|_{M_{p_\theta(\cdot)}}^{1-\theta}
\|b_{n_0}*\varphi_\delta-b_{n_0}\|_{L^\infty(\R)}^\theta.
\label{eq:vanishing-multipliers-8}
\end{align}
Combining \eqref{eq:vanishing-multipliers-7} and 
\eqref{eq:vanishing-multipliers-8}, we conclude that there exists $\delta_0>0$
such that
\begin{equation}\label{eq:vanishing-multipliers-9}
\|b_{n_0}*\varphi_{\delta_0}-b_{n_0}\|_{M_{p(\cdot)}}<\eps/2.
\end{equation} 
Hence, it follows from \eqref{eq:vanishing-multipliers-5} and 
\eqref{eq:vanishing-multipliers-9} that for every $\eps>0$ there exists
a function $b_{n_0}*\varphi_{\delta_0}\in C_c^\infty(\R)$ such that
\begin{equation}\label{eq:vanishing-multipliers-10}
\|a-b_{n_0}*\varphi_{\delta_0}\|_{M_{p(\cdot)}}<\eps.
\end{equation}
Thus $a\in\operatorname{clos}_{M_{p(\cdot)}}\big(C_c^\infty(\R)\big)$.
Part (a) is proved.

(b) Let $a\in C_0(\R)\cap V(\R)$. Fix $\eps>0$. By 
Theorem~\ref{th:Deining-interpolation}, there exist two constants
$p_0\in(1,\infty)$, $\theta\in(0,1)$ and a variable exponent 
$p_\theta(\cdot)\in\cB_M(\R)$ such that \eqref{eq:Diening-interpolation} is
fulfilled. If $2\le p_0<\infty$, then take $q\in(p_0,\infty)$.
If $1<p_0<2$, then take $q\in(1,p_0)$. In both cases, choose $\eta$ such that
\begin{equation}\label{eq:vanishing-multipliers-11}
\frac{1}{p_0}=\frac{\eta}{2}+\frac{1-\eta}{q}.
\end{equation}
Then, as it is easily seen, in both cases we have
\[
\eta=\frac{2p_0-2q}{2p_0-p_0q}\in(0,1].
\]
Take $b_n:=a\psi_n$. Then it follows from \eqref{eq:vanishing-multipliers-2}--%
\eqref{eq:vanishing-multipliers-3} and the Stechkin type inequality 
\eqref{eq:Stechkin} that 
\begin{align}
\label{eq:vanishing-multipliers-12}
&
\|a-b_n\|_{M_{p_\theta(\cdot)}}
\le 
(1+c_\theta)\|a\|_{M_{p_\theta(\cdot)}} 
\le 
(1+c_\theta)c_\theta\|a\|_{V(\R)},
\\
\label{eq:vanishing-multipliers-13}
&
\|b_n\|_{M_{p_\theta(\cdot)}}
\le 
c_\theta\|a\|_{M_{p_\theta(\cdot)}}
\le 
c_\theta^2\|a\|_{V(\R)}
\end{align}
and
\begin{align}
\label{eq:vanishing-multipliers-14}
&
\|a-b_n\|_{M_q}
\le 
\left(1+\|\psi_n\|_{M_q}\right)\|a\|_{M_q}
\le 
(1+c_q)c_q\|a\|_{V(\R)},
\\
\label{eq:vanishing-multipliers-15}
&
\|b_n\|_{M_q}
\le 
\|\psi_n\|_{M_q}\|a\|_{M_q}
\le 
c_q^2\|a\|_{V(\R)},
\end{align}
where $c_q:=3\|S\|_{\cB(L^q(\R))}$.

Equalities \eqref{eq:Diening-interpolation} and 
\eqref{eq:vanishing-multipliers-11}, Theorem~\ref{th:interpolation}, and
inequalities \eqref{eq:vanishing-multipliers-12} and
\eqref{eq:vanishing-multipliers-14} imply that for every $n\in\N$,
\begin{align}
\|a-b_n\|_{M_{p(\cdot)}}
& \le 
4\|a-b_n\|_{M_{p_0}}^\theta
\|a-b_n\|_{M_{p_\theta(\cdot)}}^{1-\theta}
\nonumber
\\
&\le 
4
\left[
\|a-b_n\|_{L^\infty(\R)}^\eta 
\|a-b_n\|_{M_q}^{1-\eta}
\right]^\theta
\left[(1+c_\theta)c_\theta\|a\|_{V(\R)}\right]^{1-\theta}
\nonumber 
\\
&\le
4\big[(1+c_\theta)c_\theta\big]^{1-\theta}
\big[(1+c_q)c_q\big]^{(1-\eta)\theta}
\|a\|_{V(\R)}^{(1-\eta)\theta+1-\theta}
\|a-b_n\|_{L^\infty(\R)}^{\eta\theta}.
\label{eq:vanishing-multipliers-16}
\end{align}
It follows from \eqref{eq:vanishing-multipliers-1} and
\eqref{eq:vanishing-multipliers-16} that there exists $n_0\in\N$
such that \eqref{eq:vanishing-multipliers-5} holds.

As before, equalities \eqref{eq:Diening-interpolation} and 
\eqref{eq:vanishing-multipliers-11}, Theorem~\ref{th:interpolation}, and
inequalities \eqref{eq:vanishing-multipliers-6},
\eqref{eq:vanishing-multipliers-13}, \eqref{eq:vanishing-multipliers-15} and
\[
\|b_{n_0}*\varphi_\delta\|_{M_q}
\le
\|b_{n_0}\|_{M_q},
\quad\delta>0,
\]
(obtained by analogy with \eqref{eq:vanishing-multipliers-6})
imply that for every $\delta>0$,
\begin{align}
\|b_{n_0}*\varphi_\delta- b_{n_0}\|_{M_{p(\cdot)}}
&
\le 
4\|b_{n_0}*\varphi_\delta-b_{n_0}\|_{M_{p_0}}^\theta
\|b_{n_0}*\varphi_\delta-b_{n_0}\|_{M_{p_\theta(\cdot)}}^{1-\theta}
\nonumber
\\
&
\le 
4\left[
\|b_{n_0}*\varphi_\delta-b_{n_0}\|_{L^\infty(\R)}^\eta 
\left(2\|b_{n_0}\|_{M_q}\right)^{1-\eta}
\right]^\theta
\left[2\|b_{n_0}\|_{M_{p_\theta(\cdot)}}\right]^{1-\theta}
\nonumber
\\
&
\le 
2^{2+(1-\eta)\theta+1-\theta}c_q^{2(1-\eta)\theta}c_\theta^{2(1-\theta)}
\|a\|_{V(\R)}^{(1-\eta)\theta+1-\theta}
\|b_{n_0}*\varphi_\delta-b_{n_0}\|_{L^\infty(\R)}^{\eta\theta}.
\label{eq:vanishing-multipliers-17}
\end{align}
Combining \eqref{eq:vanishing-multipliers-7} and 
\eqref{eq:vanishing-multipliers-17}, we conclude that there exists
$\delta_0>0$ such that \eqref{eq:vanishing-multipliers-9} holds.
Hence it follows from \eqref{eq:vanishing-multipliers-5} and
\eqref{eq:vanishing-multipliers-9} that for every $a\in C_0(\R)\cap V(\R)$
and every $\eps>0$ there exists a function 
$b_{n_0}*\varphi_{\delta_0}\in C_c^\infty(\R)$ such that 
\eqref{eq:vanishing-multipliers-10} is fulfilled. Thus
$a\in\operatorname{clos}_{M_{p(\cdot)}}\big(C_c^\infty(\R)\big)$.
\qed
\end{proof}
\subsection{Continuous Fourier multipliers on one-point and two-point
compactifications of the real line}
For a function $f\in C(\overline{\R})$, let 
\begin{equation}\label{eq:jump-killer-at-infinity}
J_f(\infty,x):=\left\{\begin{array}{lll}
f(-\infty) & \mbox{if} & x\in(-\infty,-1),
\\
\frac{1}{2}\big[f(-\infty)(1-x)+f(+\infty)(1+x)\big],
& \mbox{if} &x\in[-1,1],
\\
f(+\infty), & \mbox{if} & x\in(1,+\infty).
\end{array}\right.
\end{equation}
It is easy to see that
\begin{equation}\label{eq:norm-of-jump-killer-at-infinity}
\|J_f(\infty,\cdot)\|_{V(\R)}
=
\max\big\{|f(-\infty)|,|f(+\infty)|\big\}
+
|f(+\infty)-f(-\infty)|.
\end{equation}
Therefore $J_f(\infty,\cdot)\in C(\overline{\R})\cap V(\R)$ and
$f\in J_f(\infty,\cdot)\in C_0(\R)$.

The following lemma is proved by analogy with \cite[Lemma~3.1(i)]{KS94}.
\begin{lemma}\label{le:multipliers-one-two-point}
If $p(\cdot)\in\cB_M(\R)$, then 
\begin{equation}\label{eq:multipliers-one-two-point-1}
C_{p(\cdot)}(\dR)= C(\dR)\cap C_{p(\cdot)}(\overline{\R}).
\end{equation}
\end{lemma}
\begin{proof}
It is obvious that $C_{p(\cdot)}(\dR)\subset C_{p(\cdot)}(\overline{\R})$.
On the other hand, it follows from 
Theorem~\ref{th:continuous-embedding-of-multipliers} that  
$C_{p(\cdot)}(\dR)\subset C(\dR)$. Therefore,
\begin{equation}\label{eq:multipliers-one-two-point-2}
C_{p(\cdot)}(\dR)\subset C(\dR)\cap C_{p(\cdot)}(\overline{\R}).
\end{equation}

To prove the opposite embedding, let us consider an arbitrary function
$a$ in $C_{p(\cdot)}(\overline{\R})$ such that $a(+\infty)=a(-\infty)$.
Let $\{a_n\}_{n\in\N}\subset C(\overline{\R})\cap V(\R)$ be a sequence
such that $\|a_n-a\|_{M_{p(\cdot)}}\to 0$ as $n\to\infty$. According to
Theorem~\ref{th:continuous-embedding-of-multipliers}, the sequence
$\{a_n\}_{n\in\N}$ converges to $a$ uniformly on $\R$. Hence, in particular, 
$a_n(\pm\infty)\to a(\infty)$ as $n\to\infty$. Let the functions 
$b_n:=J_{a_n-a(\infty)}(\infty,\cdot)\in C(\overline{\R})\cap V(\R)$ be 
defined by \eqref{eq:jump-killer-at-infinity} with $a_n-a(\infty)$ in place 
of $f$. By the Stechkin type inequality \eqref{eq:Stechkin} and equality
\eqref{eq:norm-of-jump-killer-at-infinity}, we have
\begin{align*}
\|b_n\|_{M_{p(\cdot)}}
\le& 
\|S\|_{\cB(L^{p(\cdot)}(\R))}
\|J_{a_n-a(\infty)}(\infty,\cdot)\|_{V(\R)}
\\
=&
\|S\|_{\cB(L^{p(\cdot)}(\R))}
\max\big\{|a_n(-\infty)-a(\infty)|,|a_n(+\infty)-a(\infty)|\big\}
+
\|S\|_{\cB(L^{p(\cdot)}(\R))}
|a_n(+\infty)-a_n(-\infty)|.
\end{align*}
Therefore, $\|b_n\|_{M_{p(\cdot)}}\to 0$ 
as $n\to\infty$ and thus,
\[
\lim_{n\to\infty}\|a_n-b_n-a\|_{M_{p(\cdot)}}=0.
\]
Since $a_n-b_n\in C(\dR)\cap V(\R)$, the latter equality implies that
$a\in C_{p(\cdot)}(\dR)$. Thus
\begin{equation}\label{eq:multipliers-one-two-point-3}
C(\dR)\cap C_{p(\cdot)}(\overline{\R})\subset C_{p(\cdot)}(\dR).
\end{equation}
Combining embeddings \eqref{eq:multipliers-one-two-point-2}--%
\eqref{eq:multipliers-one-two-point-3}, we arrive at
\eqref{eq:multipliers-one-two-point-1}.
\qed
\end{proof}
\subsection{Proof of Theorem~\ref{th:first}}
Since $\C{\bf{\dot{+}}} C_c^\infty(\R)\subset C(\dR)\cap V(\R)$, we have
\begin{equation}\label{eq:first-1}
\operatorname{clos}_{M_{p(\cdot)}}\big(\C{\bf{\dot{+}}} C_c^\infty(\R)\big)
\subset 
\operatorname{clos}_{M_{p(\cdot)}}\big(C(\dR)\cap V(\R)\big)
=
C_{p(\cdot)}(\dR).
\end{equation}
To prove the opposite embedding, take $a\in C_{p(\cdot)}(\dR)$ and fix 
$\eps>0$. Then there exists $b\in C(\dR)\cap V(\R)$ such that
\begin{equation}\label{eq:first-2}
\|a-b\|_{M_{p(\cdot)}}<\eps/2.
\end{equation}
Then $b-b(\infty)\in C_0(\R)\cap V(\R)$ and in view of
Theorem~\ref{th:vanishing-multipliers}(b) there exists $c\in C_c^\infty(\R)$
such that 
\begin{equation}\label{eq:first-3}
\|b-b(\infty)-c\|_{M_{p(\cdot)}}<\eps/2.
\end{equation}
It follows from inequalities \eqref{eq:first-2}--\eqref{eq:first-3} that
\begin{equation}\label{eq:first-4}
\|a-(b(\infty)+c)\|_{M_{p(\cdot)}}<\eps.
\end{equation}
Since $b(\infty)+c\in\C{\bf{\dot{+}}} C_c^\infty(\R)$, the latter inequality
implies that the function $a$ belongs to
$\operatorname{clos}_{M_{p(\cdot)}}
\big(\C{\bf{\dot{+}}} C_c^\infty(\R)\big)$. 
Hence
\begin{equation}\label{eq:first-5}
C_{p(\cdot)}(\dR)\subset \operatorname{clos}_{M_{p(\cdot)}}
\big(\C{\bf{\dot{+}}} C_c^\infty(\R)\big).
\end{equation}
Combining embeddings \eqref{eq:first-1} and \eqref{eq:first-5}
with Lemma~\ref{le:multipliers-one-two-point}, we arrive at the statement of
Theorem~\ref{th:first}.
\qed
\subsection{Proof of Theorem~\ref{th:second}}
By the Stechkin type inequality \eqref{eq:Stechkin}, for every 
$\theta\in(0,\tau_{p(\cdot)})$,
\[
C(\dR)\cap V(\R)\subset C(\dR)\cap M_{p_\theta(\cdot)},
\quad
C(\overline{\R})\cap V(\R)
\subset
C(\overline{\R})\cap M_{p_\theta(\cdot)},
\]
whence
\begin{eqnarray}
\label{eq:second-1}
&&
C_{p(\cdot)}(\dR)
=
\operatorname{clos}_{M_{p(\cdot)}} \big(C(\dR)\cap V(\R)\big)
\subset
\operatorname{clos}_{M_{p(\cdot)}} 
\big(C(\dR)\cap M_{\langle p(\cdot)\rangle}\big)
=
C_{\langle p(\cdot)\rangle}(\dR),
\\
\label{eq:second-2}
&&
C_{p(\cdot)}(\overline{\R})
=
\operatorname{clos}_{M_{p(\cdot)}} \big(C(\overline{\R})\cap V(\R)\big)
\subset
\operatorname{clos}_{M_{p(\cdot)}} 
\big(C(\overline{\R})\cap M_{\langle p(\cdot)\rangle}\big)
=
C_{\langle p(\cdot)\rangle}(\overline{\R}),
\end{eqnarray}

To prove the opposite embedding to \eqref{eq:second-1}, consider
$a\in C_{\langle p(\cdot)\rangle}(\dR)$. Then for every $\eps>0$ there exists
$b\in C(\dR)\cap M_{\langle p(\cdot)\rangle}$ such that \eqref{eq:first-2}
holds. Since $b(\infty)\in\C\subset M_{\langle p(\cdot)\rangle}$, we see that
$b-b(\infty)\in C_0(\R)\cap M_{\langle p(\cdot)\rangle}$. In view of
Theorem~\ref{th:vanishing-multipliers}(a), there exists $c\in C_c^\infty(\R)$
such that \eqref{eq:first-3} is fulfilled. It follows from inequalities
\eqref{eq:first-2}--\eqref{eq:first-3} that inequality \eqref{eq:first-4}
holds. Since 
$b(\infty)+c\in\C{\bf{\dot{+}}} C_c^\infty(\R)\subset C(\dR)\cap V(\R)$,
inequality \eqref{eq:first-4} implies that $a\in C_{p(\cdot)}(\dR)$. Hence
\begin{equation}\label{eq:second-3}
C_{\langle p(\cdot)\rangle}(\dR)\subset C_{p(\cdot)}(\dR).
\end{equation}
Combining embeddings \eqref{eq:second-1} and \eqref{eq:second-3}, we arrive
at the first equality in Theorem~\ref{th:second}.

Now assume that $a\in C_{\langle p(\cdot)\rangle}(\overline{\R})$. Then for 
every $\eps>0$ there exists a function 
$b\in C(\overline{\R})\cap M_{\langle p(\cdot)\rangle}$ 
such that \eqref{eq:first-2} holds. Let $J_b(\infty,\cdot)$ be defined by
\eqref{eq:jump-killer-at-infinity} with $b$ in place of $f$. Then, by the 
Stechkin type inequality \eqref{eq:Stechkin}, 
$J_b(\infty,\cdot)\in V(\R)\subset M_{\langle p(\cdot)\rangle}$.
Therefore, $b-J_b(\infty,\cdot)\in C_0(\R)\cap M_{\langle p(\cdot)\rangle}$.
By Theorem~\ref{th:vanishing-multipliers}(a), there exists a function 
$c\in C_c^\infty(\R)$ such that 
\begin{equation}\label{eq:second-4}
\|b-J_b(\infty,\cdot)-c\|_{M_{p(\cdot)}}<\eps/2.
\end{equation}
It follows from inequalities \eqref{eq:first-2} and \eqref{eq:second-4}
that 
\[
\|a-(J_b(\infty,\cdot)+c)\|_{M_{p(\cdot)}}<\eps.
\]
Since $J_b(\infty,\cdot)\in C(\overline{\R})\cap V(\R)$ and
$c\in C_c^\infty(\R)\subset C(\overline{\R})\cap V(\R)$, the latter
inequality implies that the function $a$ belongs to
$\operatorname{clos}_{M_{p(\cdot)}}\big(C(\overline{\R})\cap V(\R)\big)=
C_{p(\cdot)}(\overline{\R})$. Hence
\begin{equation}\label{eq:second-5}
C_{\langle p(\cdot)\rangle}(\overline{\R})\subset C_{p(\cdot)}(\overline{\R}).
\end{equation}
Combining embeddings \eqref{eq:second-2} and \eqref{eq:second-5}, we arrive
at the second equality in Theorem~\ref{th:second}.
\qed
\subsection{Proof of Corollary~\ref{co:standard-Lp}}
Let $c\in\C$ and $\varphi\in C_c^\infty(\R)$. Then we have
$F^{-1}\varphi\in\mathcal{S}(\R)\subset L^1(\R)$. Hence $c+\varphi\in W(\R)$ 
and $\C{\bf{\dot{+}}}C_c^\infty(\R)\subset W(\R)$. Then, taking into
account Theorem~\ref{th:first}, we see that
\begin{equation}\label{eq:standard-Lp-1}
C_p(\dR)=\operatorname{clos}_{M_p}\big(\C{\bf{\dot{+}}}C_c^\infty(\R)\big)
\subset\operatorname{clos}_{M_p}\big(W(\R)\big).
\end{equation}
It is shown in the proof of \cite[Theorem~2.13]{D79} that for every 
$a\in W(\R)$ and $\eps>0$ there exists a function $b$ of the form 
\[
b(x)=c-\sum_{k=-n}^n c_k-\sum_{k=-n}^n c_k\left(\frac{x-i}{x+i}\right)^k,
\quad x\in\R,
\]
where $c,c_{-n},\dots,c_n\in\C$ and $n\in\N\cup\{0\}$, such that 
$\|a-b\|_{M_p}<\eps$. It is easy to check that $b\in C(\dR)\cap V(\R)$.
Hence $a\in\operatorname{clos}_{M_p}\big(C(\dR)\cap V(\R)\big)=C_p(\dR)$
and
\begin{equation}\label{eq:standard-Lp-2}
\operatorname{clos}_{M_p}\big(W(\R)\big)\subset C_p(\dR).
\end{equation}
Combining \eqref{eq:standard-Lp-1}--\eqref{eq:standard-Lp-2}
and the first equality in Theorem~\ref{th:second} for the constant exponent
$p\in(1,\infty)\setminus\{2\}$, we arrive at the equality
\[
C_p(\dR)=\operatorname{clos}_{M_p}\big(W(\R)\big)=C_{\langle p\rangle}(\dR).
\]
By \cite[Theorem~A]{FG71}, there exists a function 
$a\in C(\dR)\cap M_p(\R)\setminus\operatorname{clos}_{M_p}\big(W(\R)\big)$,
which completes the proof of 
$C_{\langle p\rangle}(\dR)\subsetneqq C(\dR)\cap M_p$.
It remains to observe that the equality 
$C_p(\overline{\R})=C_{\langle p\rangle}(\overline{\R})$
follows immediately from the second equality in Theorem~\ref{th:second}.
\qed
\section{Final remarks and open problems}\label{sec:final}
\subsection{Embedding of the algebra $C_{p(\cdot)}(\dR)$ into the algebra 
$SO_{p(\cdot)}$ of slowly oscillating Fourier multipliers} 
Let $C_b(\R):=C(\R)\cap L^\infty(\R)$. For a bounded measurable function
$f:\R\to\C$ and a set $J\subset\R$, let
\[
\operatorname{osc}(f,J):=
\operatornamewithlimits{ess\,sup}_{x,y\in J}|f(x)-f(y)|.
\]
Let $SO$ be the $C^*$-algebra of all slowly oscillating functions at $\infty$ 
defined by
\[
SO:=\left\{f\in C_b(\R): \lim_{x\to +\infty}
\operatorname{osc}(f,[-x,-x/2]\cup[x/2,x])=0\right\}.
\]
Consider the differential operator $(Df)(x)=xf'(x)$ and its iterations defined
by $D^0f=f$ and $D^jf=D(D^{j-1}f)$ for $j\in\N$. Let
\[
SO^3:=\left\{a\in SO\cap C^3(\R):\lim_{x\to\infty} (D^ja)(x)=0,\ j=1,2,3
\right\},
\]
where $C^3(\R)$ denotes the set of all three times continuously differentiable
functions. It is easy to see that $SO^3$ is a commutative Banach algebra under
pointwise operations and the norm
\[
\|a\|_{SO^3}:=\sum_{j=0}^3\frac{1}{j!}\|D^ja\|_{L^\infty(\R)}.
\]
It follows from \cite[Corollary~2.8]{K15c} that if $p(\cdot)\in\cB_M(\R)$,
then there exists a positive constant $c_{p(\cdot)}$ depending only on
the variable exponent $p(\cdot)$ such that for every function $a\in SO^3$,
\[
\|a\|_{M_{p(\cdot)}}\le c_{p(\cdot)}\|a\|_{SO^3}.
\]
For $p(\cdot)\in\cB_M(\R)$, consider the algebra of slowly oscillating at
$\infty$ Fourier multipliers defined by
\[
SO_{p(\cdot)}:=\operatorname{clos}_{M_{p(\cdot)}}\big(SO^3\big).
\]

Since $\C{\bf{\dot{+}}} C_c^\infty(\R)\subset SO^3$, Theorem~\ref{th:first} 
yields the following analogue of \cite[Lemma~3.6]{KJLH09}.
\begin{corollary}
If $p(\cdot)\in\cB_M(\R)$, then $C_{p(\cdot)}(\dR)\subset SO_{p(\cdot)}$.
\end{corollary}

\subsection{On relations between $PC_{p(\cdot)}$, 
$PC_{\langle p(\cdot)\rangle}^0$ and $PC_{\langle p(\cdot)\rangle}$}
Let $PC$ be the $C^*$-algebra of all bounded piecewise continuous functions on 
$\dR$. By definition, $a\in PC$ if and only if $a\in L^\infty(\R)$ and the
finite one-sided limits
\[
a(x_0-0):=\lim_{x\to x_0-0}a(x),
\quad
a(x_0+0):=\lim_{x\to x_0+0}a(x)
\] 
exist for each $x_0\in\dR$. Let $PC^0$ (resp. $P\C^0$) denote the set of all 
piecewise continuous functions with finitely many jumps (resp. piecewise
constant functions with finitely many jumps). It is clear that 
$P\C^0\subset V(\R)$. By \cite[Lemma~2.10]{D79}, $V(\R)\subset PC$.
For $p(\cdot)\in\cB_M(\R)$, consider the algebra
\[
PC_{p(\cdot)} 
:=
\operatorname{clos}_{M_{p(\cdot)}} \big( V(\R) \big).
\]
\begin{theorem}\label{th:PC-multipliers}
If $p(\cdot)\in\cB_M(\R)$, then 
\begin{equation}\label{eq:PC-multipliers-1}
PC_{p(\cdot)}=\operatorname{clos}_{M_{p(\cdot)}}\big(P\C^0\big).
\end{equation}
\end{theorem}
\begin{proof}
Since $P\C^0\subset V(\R)$, we obviously have
\begin{equation}\label{eq:PC-multipliers-2}
\operatorname{clos}_{M_{p(\cdot)}}\big(P\C^0\big)\subset PC_{p(\cdot)}.
\end{equation}

Let $a\in PC_{p(\cdot)}$ and $\eps>0$. Then there exists $b\in V(\R)$
such that 
\begin{equation}\label{eq:PC-multipliers-3}
\|a-b\|_{M_{p(\cdot)}}<\eps/2.
\end{equation}
By \cite[Lemma~2.10]{D79}, there exists a sequence 
$\{b_n\}_{n\in\N}\subset P\C^0$ such that
\begin{equation}\label{eq:PC-multipliers-4}
\lim_{n\to\infty}\|b_n-b\|_{L^\infty(\R)},
\quad
\sup_{n\in\N}V(b_n)\le V(b).
\end{equation}
Then there exists $N\in\N$ such that
\begin{equation}\label{eq:PC-multipliers-5}
\sup_{n\ge N}\|b_n\|_{V(\R)}\le 2\|b\|_{V(\R)}.
\end{equation}

Let $p_0\in(1,\infty)$, $\theta\in(0,1)$, $\eta\in(0,1]$ and 
$p_\theta(\cdot)\in\cB_M(\R)$ be as in the proof of 
Theorem~\ref{th:vanishing-multipliers}(b). It follows from the Stechkin
type inequality \eqref{eq:Stechkin} and inequality \eqref{eq:PC-multipliers-5}
that for all $n\ge N$,
\begin{align}
\label{eq:PC-multipliers-6}
&
\|b_n-b\|_{M_{p_\theta(\cdot)}}
\le
\|b_n\|_{M_{p_\theta(\cdot)}}+\|b\|_{M_{p_\theta(\cdot)}}
\le 
c_\theta\|b\|_{V(\R)},
\\
\label{eq:PC-multipliers-7}
&
\|b_n-b\|_{M_q}
\le
\|b_n\|_{M_q}+\|b\|_{M_q}
\le 
c_q\|b\|_{V(\R)},
\end{align}
where $c_\theta:=3\|S\|_{\cB(L^{p_\theta(\cdot)}(\R))}$
and $c_q:=3\|S\|_{\cB(L^q(\R))}$.

Equalities \eqref{eq:Diening-interpolation} and 
\eqref{eq:vanishing-multipliers-11},
Theorem~\ref{th:interpolation} and inequalities 
\eqref{eq:PC-multipliers-6}--\eqref{eq:PC-multipliers-7} imply that for every 
$n\ge N$,
\begin{align}\label{eq:PC-multipliers-8}
\|b_n-b\|_{M_{p(\cdot)}}
&\le 
4\|b_n-b\|_{M_{p_0}}^\theta
\|b_n-b\|_{M_{p_\theta(\cdot)}}^{1-\theta}
\\
\nonumber
&\le 
4\left(
\|b_n-b\|_{L^\infty(\R)}^\eta
\|b_n-b\|_{M_q}^{1-\eta}
\right)^\theta
\|b_n-b\|_{M_{p_\theta(\cdot)}}^{1-\theta}
\\
\nonumber
&\le 
4c_q^{(1-\eta)\theta}c_\theta^{1-\theta}\|b\|_{V(\R)}^{(1-\eta)\theta+1-\theta}
\|b_n-b\|_{L^\infty(\R)}^{\eta\theta}.
\end{align}
It follows from the equality in \eqref{eq:PC-multipliers-4} and inequality
\eqref{eq:PC-multipliers-8} that there exists $n_0\ge N$ such that
\begin{equation}\label{eq:PC-multipliers-9}
\|b_{n_0}-b\|_{M_{p(\cdot)}}<\eps/2.
\end{equation}
Combining inequalities \eqref{eq:PC-multipliers-3} and 
\eqref{eq:PC-multipliers-9} we see that for every $\eps>0$ there exists 
$c:=b_{n_0}\in P\C^0$ such that $\|a-c\|_{M_{p(\cdot)}}<\eps$. Hence
$a\in\operatorname{clos}_{M_{p(\cdot)}}\big(P\C^0\big)$. Thus
\begin{equation}\label{eq:PC-multipliers-10}
PC_{p(\cdot)}\subset\operatorname{clos}_{M_{p(\cdot)}}\big(P\C^0\big).
\end{equation}
Embeddings \eqref{eq:PC-multipliers-2} and \eqref{eq:PC-multipliers-10}
yield equality \eqref{eq:PC-multipliers-1}.
\qed
\end{proof}

Further, for $p(\cdot)\in\cB_M^*(\R)$, consider the algebras
\[
PC_{\langle p(\cdot) \rangle}^0
:=
\operatorname{clos}_{M_{p(\cdot)}}
\big( PC^0\cap M_{\langle p(\cdot) \rangle}\big),
\quad
PC_{\langle p(\cdot) \rangle} 
:=
\operatorname{clos}_{M_{p(\cdot)}} 
\big( PC\cap M_{\langle p(\cdot) \rangle}\big).
\]
\begin{theorem}\label{th:finite-jumps-multipliers}
If $p(\cdot)\in\cB_M^*(\R)$, then
$PC_{p(\cdot)}=PC_{\langle p(\cdot)\rangle}^0
\subset PC_{\langle p(\cdot)\rangle}$. 
\end{theorem}
\begin{proof}
The embedding
\begin{equation}\label{eq:finite-jumps-multipliers-1}
PC_{\langle p(\cdot)\rangle}^0\subset PC_{\langle p(\cdot)\rangle}
\end{equation}
is obvious.

Since $P\C^0\subset PC^0\cap V(\R)\subset PC^0\cap M_{\langle p(\cdot)\rangle}$
in view of the Stechkin type inequality \eqref{eq:Stechkin}, it follows from
Theorem~\ref{th:PC-multipliers} that
\begin{equation}\label{eq:finite-jumps-multipliers-2}
PC_{p(\cdot)}\subset  PC_{\langle p(\cdot)\rangle}^0.
\end{equation}

To prove the opposite embedding, take $a\in PC_{\langle p(\cdot)\rangle}^0$ and
$\eps>0$. Then there exists $b\in PC^0\cap M_{\langle p(\cdot)\rangle}$ such
that
\begin{equation}\label{eq:finite-jumps-multipliers-3}
\|a-b\|_{M_{p(\cdot)}}<\eps/2.
\end{equation}

Let $x_1,\dots,x_m\in\dR$ be the jumps of $b$. Define the function
$J_b(\infty,\cdot)$ by \eqref{eq:jump-killer-at-infinity} with $b$
in place of $f$. For $x_0\in\{x_1,\dots,x_m\}\setminus\{\infty\}$, let
\[
J_b(x_0,x):=
\left\{\begin{array}{lll}
f(x_0+0)(x_0+1-x) &\mbox{if}& x\in(x_0,x_0+1],
\\
f(x_0-0)(x+1-x_0) &\mbox{if}& x\in[x_0-1,x_0],
\\
0 &\mbox{if}& x\in\R\setminus[x_0-1,x_0+1].
\end{array}\right.
\]
It is easy to see that the function 
\[
c:=\sum_{x_0\in\{x_1,\dots,x_m\}\setminus\{\infty\}} J_b(x_0,\cdot)
+J_b(\infty,\cdot)
\]
belongs to $V(\R)$. Hence $c\in M_{\langle p(\cdot)\rangle}$ in view of the
Stechkin type inequality \eqref{eq:Stechkin}. On the other hand, it follows
from the construction of the function $c$ that the function $d:=b-c$ belongs 
to $C_0(\R)$. Since $b,c\in M_{\langle p(\cdot)\rangle}$, we see that 
$d\in C_0(\R)\cap M_{\langle p(\cdot)\rangle}$. It follows from
Theorem~\ref{th:vanishing-multipliers}(a) that there exists 
$f\in C_c^\infty(\R)\subset V(\R)$ such that
\begin{equation}\label{eq:finite-jumps-multipliers-4}
\|b-c-f\|_{M_{p(\cdot)}}=\|d-f\|_{M_{p(\cdot)}}<\eps/2.
\end{equation}

Inequalities \eqref{eq:finite-jumps-multipliers-3} and
\eqref{eq:finite-jumps-multipliers-4} imply that 
$\|a-(c+f)\|_{M_{p(\cdot)}}<\eps$. Since $c+f\in V(\R)$, the latter inequality
implies that 
$a\in \operatorname{clos}_{M_{p(\cdot)}}\big(V(\R)\big)=PC_{p(\cdot)}$.
Thus
\begin{equation}\label{eq:finite-jumps-multipliers-5}
PC_{\langle p(\cdot)\rangle}^0\subset PC_{p(\cdot)}.
\end{equation}
Combining embeddings \eqref{eq:finite-jumps-multipliers-1},
\eqref{eq:finite-jumps-multipliers-2} and \eqref{eq:finite-jumps-multipliers-5},
we arrive at the desired statement.
\qed
\end{proof}

We have not been able to show that $PC\cap M_{\langle p(\cdot)\rangle}$
is a subset of $PC_{p(\cdot)}$ and therefore we must raise the inclusion
$PC\cap M_{\langle p(\cdot)\rangle}\subset PC_{p(\cdot)}$ and the
resulting equality $PC_{p(\cdot)}=PC_{\langle p(\cdot)\rangle}$ as an open
question even in the case of constant exponents 
$p\in(1,\infty)\setminus\{2\}$ (see also \cite[Section~6.27]{BS06} for the
discrete analogue of this open problem).
\subsection{On relations between $C_p(\overline{\R})$ and $\Pi C_p^\infty(\R)$}
Let $p\in(1,\infty)\setminus\{2\}$ be constant. 
R. Duduchava and A. Saginashvili 
\cite[\S 1.3]{DS81} (see also \cite[p.~124]{KS94}) defined the algebra
\[
\Pi C_p^\infty(\R):=C(\overline{\R})
\cap 
\operatorname{clos}_{M_p}\big(M_{\langle p\rangle}\big),
\]
where $M_{\langle p\rangle}$ is defined by \eqref{eq:M-p-cloud}. 
It follows from Corollary~\ref{co:standard-Lp} and the continuous
embedding $M_p\subset L^\infty(\R)$ 
(see, e.g., Theorem~\ref{th:continuous-embedding-of-multipliers})
that
\[
C_p(\overline{\R})
=
\operatorname{clos}_{M_p}
\big(C(\overline{\R})\cap M_{\langle p\rangle}\big)
\subset
C(\overline{\R})
\cap 
\operatorname{clos}_{M_p}\big(M_{\langle p\rangle}\big)
=\Pi C_p^\infty(\R).
\]
In \cite[Lemma~3.1~(ii)]{KS94} (see also \cite[p.~385]{BKS02}) it is claimed 
that  $C_p(\overline{\R})=\Pi C_p^\infty(\R)$. However, the proof of the 
embedding
\begin{equation}\label{eq:Karlovich-Spitkovsky-open}
\Pi C_p^\infty(\R)\subset C_p(\overline{\R})
\end{equation}
in \cite{KS94} contains a gap because it is not explained why one can 
approximate a function $a\in \Pi C_p^\infty(\R)$ in the norm of $M_p$ by 
functions  $a_n\in C(\overline{\R})\cap  M_{r_n}$, where $r_n\in R_p$ for all 
$n\in\N$ and $R_p$ is given by \eqref{eq:Rp}. Since we have not been able 
either to prove or to disprove  \eqref{eq:Karlovich-Spitkovsky-open}, 
we left it as an open problem.
\subsection{On relations between $\cB_M^*(\R)$ and $\cB_M(\R)$}
As it was already mentioned,  $LH(\R)\subset\cB_M^*(\R)\subset\cB_M(\R)$ and 
$LH(\R)\subsetneqq\cB_M(\R)$. We have not been able to prove or to disprove
the equality
\[
\cB_M^*(\R)=\cB_M(\R).
\]
We left it as an open problem.

\bigskip
\textbf{Acknowledgments.} 
This work was partially supported by the Funda\c{c}\~ao para a Ci\^encia e a
Tecnologia (Portu\-guese Foundation for Science and Technology)
through the project
UID/MAT/00297/2019 (Centro de Matem\'atica e Aplica\c{c}\~oes).

The author would like to thank Eugene Shargorodsky (King's College London, UK)
for interesting discussions on the subject of the paper.



\begin{thebibliography}{99}
\bibitem{BBK04}
M. A. Bastos, A. Bravo, Yu. I. Karlovich,
Convolution type operators with symbols generated by slowly
oscillating and piecewise continuous matrix functions,
Oper. Theory: Adv. Appl. \textbf{147} (2004), 151--174.

\bibitem{BS88} 
C. Bennett, R. Sharpley, 
Interpolation of Operators, 
Academic Press, Boston, 1988.

\bibitem{BKS02}
A. B\"ottcher, Yu. I. Karlovich, I. M. Spitkovsky, 
Convolution Operators and Factorization of Almost Periodic Matrix 
Functions,
Birkh\"auser, Basel, 2002.

\bibitem{BS06}
A. B\"ottcher,  B. Silbermann,
Analysis of Toeplitz Operators,
2nd ed., Springer, Berlin, 2006.

\bibitem{B11}
H. Brezis, 
Functional Analysis, Sobolev Spaces and Partial Differential 
Equations,
Springer, New York, 2011.

\bibitem{CF13}
D. Cruz-Uribe, A. Fiorenza, 
Variable Lebesgue spaces,
Birkh\"auser/Springer, New York, 2013.

\bibitem{CFN03}
D. Cruz-Uribe, A. Fiorenza, C. J. Neugebauer,
The maximal function on variable $L^{p}$ spaces,
Ann. Acad. Sci. Fenn. Math. \textbf{28} (2003), 223--238.

\bibitem{CFN04}
D. Cruz-Uribe, A. Fiorenza, C. J. Neugebauer,
Corrections to: ``The maximal function on variable $L^{p}$ spaces",
Ann. Acad. Sci. Fenn. Math. \textbf{29} (2004), 247--249.

\bibitem{D04}
L. Diening,
Maximal function on generalized Lebesgue spaces $L^{p(\cdot)}$,
Math. Inequal. Appl. \textbf{7} (2004), 245--253.

\bibitem{DHHR11}
L. Diening, P. Harjulehto, P. H\"ast\"o, M. R\r{u}\v zi\v cka,
Lebesgue and Sobolev Spaces with Variable Exponents,
Springer, Berlin, 2011.

\bibitem{D79}
R. V. Duduchava,
Integral Equations with Fixed Singularities,
Teubner, Leipzig, 1979.

\bibitem{DS81}
R. V. Duduchava, A. I. Saginashvili, 
Convolution integral equations on a half-line with 
semi-almost-periodic presymbols, 
Differ. Equations \textbf{17} (1981), 207--216.

\bibitem{EG77}
R. E. Edwards, G. I. Gaudry, 
Littlewood-Paley and multiplier theory,
Springer, Berlin, 1977.

\bibitem{FG71}
A. Fig\`a-Talamanca, G. Gaudry,
Multipliers of $L^p$ which vanish at infinity,
J. Funct. Anal. \textbf{7} (1971), 475--486.

\bibitem{G14}
L. Grafakos,
Classical Fourier Analysis. 3rd ed. 
Springer, New York, 2014.

\bibitem{H60}
L. H\"ormander,
Estimates for translation invariant operators in $L^p$ spaces,
Acta Math. \textbf{104} (1960), 93--140.

\bibitem{K15a}
A. Yu. Karlovich, 
The Stechkin inequality for Fourier multipliers on variable Lebesgue 
spaces,
Math. Inequal. Appl. \textbf{18} (2015), 1473--1481.

\bibitem{K15b}
A. Yu. Karlovich,
Banach algebra of the Fourier multipliers on weighted Banach function 
spaces,
Concr. Oper. \textbf{2} (2015), 27--36.

\bibitem{K15c}
A. Yu. Karlovich, 
Commutators of convolution type operators on some Banach function 
spaces,
Ann. Funct. Anal. AFA \textbf{6} (2015), 191--205.

\bibitem{KS-PLMS}
A. Karlovich,  E. Shargorodsky, 
When does the norm of a Fourier multiplier dominate its $L^\infty$ 
norm?
Proc. London Math. Soc., doi: 10.1112/plms.12206 (2018).

\bibitem{KS13}
A. Yu. Karlovich, I. M. Spitkovsky,
Pseudodifferential operators on variable Lebesgue spaces,
Oper. Theor. Adv. Appl. \textbf{228} (2013), 173--183.

\bibitem{KS14}
A. Yu. Karlovich, I. M. Spitkovsky, 
The Cauchy singular integral operator on weighted variable Lebesgue 
spaces,
Oper. Theor. Adv. Appl. \textbf{236} (2014), 275--291.

\bibitem{KJLH09}
Yu. I. Karlovich, J. Loreto Hern\'andez,
Wiener-Hopf operators with slowly oscillating matrix symbols on 
weighted Lebesgue spaces,
Integr. Equ. Oper. Theory \textbf{64} (2009), 203--237.

\bibitem{KS94}
Yu. I. Karlovich, I. M. Spitkovsky,
(Semi)-Fredholmness of convolution operators on the spaces of Bessel
potentials,
Oper. Theor. Adv. Appl. \textbf{71} (1994), 122--152.

\bibitem{L05}
A. K. Lerner, 
Some remarks on the Hardy-Littlewood maximal
function on variable $L^p$ spaces,
Math. Z. \textbf{251} (2005), 509--521.

\bibitem{MSS14}
H. Mascarenhas, P. A. Santos, M. Seidel,
Quasi-banded operators, convolutions with almost periodic or 
quasi-continuous data, and their approximations,
J. Math. Anal. Appl. \textbf{418} (2014), 938--963.

\bibitem{MSS17}
H. Mascarenhas, P. A. Santos, M. Seidel,
Approximation sequences to operators on Banach spaces: a rich 
approach,
J. Lond. Math. Soc., II. Ser. \textbf{96} (2017), 86--110.

\bibitem{M83}
J. Musielak, 
Orlicz Spaces and Modular Spaces,
Springer, Berlin, 1983.

\bibitem{RS08}
V. Rabinovich, S. Samko, 
Boundedness and Fredholmness of pseudodifferential operators in 
variable exponent spaces,
Integr. Equ. Oper. Theor. \textbf{60} (2008), 507--537.

\bibitem{SM86}
I. B. Simonenko, Chin Ngok Min, 
Local Method in the Theory of One-Dimensional Singular Integral 
Equations with Piecewise Continuous Coefficients. Noetherity,
Rostov Univ. Press, Rostov on Don, 1986 (in Russian).
\end{thebibliography}
\end{document}